\documentclass[12pt]{amsart}%{article}
\usepackage[margin=1in]{geometry}
\usepackage[leqno]{amsmath}
\usepackage{amssymb}
\usepackage{amsthm}
\usepackage{nicematrix}
\usepackage{anyfontsize}
\usepackage{hyperref}
\usepackage{ytableau}
\usepackage{youngtab}
\usepackage{quiver}

\newtheorem{theorem}{Theorem}[section]
\newtheorem{corollary}[theorem]{Corollary}
\newtheorem{proposition}[theorem]{Proposition}
\newtheorem{lemma}[theorem]{Lemma}

\newtheorem*{theorem*}{Theorem}

\theoremstyle{remark}

\theoremstyle{definition}

\newtheorem{example}[theorem]{Example}
\newtheorem*{definition*}{Definition}

\newcommand{\C}{\mathbb{C}}
\newcommand{\Z}{\mathbb{Z}}
\newcommand{\Q}{\mathbb{Q}}

\newcommand{\ox}{\otimes}
\DeclareMathOperator{\Pf}{Pf}

\newcommand{\bibDate}[1]{}

\begin{document}
\title[Deformation of the Bernstein Operator]{Constructing Hall-Littlewood Functions via a Deformation of the Bernstein Operator}

\author{John~Graf}
\address{Department of Mathematical Sciences, University of Delaware, Newark, DE 19716, USA}
\email[John~Graf]{jrgraf@udel.edu}

% \date{March 27, 2025} 
    
    \begin{abstract}
        The Bernstein operator $\mathbf{B}_n$ acts on a Schur function $S_\lambda$ by appending a part to the index, i.e., $\mathbf{B}_n S_\lambda=S_{(n,\lambda)}$. This provides a method of constructing the vertex operator representation of Schur functions since its homogeneous components are essentially just these Bernstein operators. Meanwhile, the Hall-Littlewood functions are an important generalization of the Schur functions, and they also have a vertex operator representation due to Jing. In this paper, we construct a $t$-analogue of the Bernstein operator, which allows for an explicit construction of the Jing operator. We show that the usual involution $\omega$ is fundamental to this construction, revealing further combinatorial structure. As an application, we use this vertex operator to prove stability of certain structure coefficients, including the Hall polynomials.
    \end{abstract}

    \maketitle
    
    % \tableofcontents

    \section{Introduction}

    The Schur functions $S_\lambda(X)$, the homogeneous symmetric functions $h_\lambda(X)$, and the elementary symmetric functions $e_\lambda(X)$, are important bases for the ring $\Lambda$ of symmetric functions. When the partition $\lambda$ is a row shape $(n)$ or a column shape $(1^n)$, then the Schur function $S_\lambda$ specializes to the homogeneous and elementary symmetric functions, respectively,
    \[S_{(n)}(X)=h_n(X),\qquad \text{and}\qquad S_{(1^n)}(X)=e_n(X).\]
    This special relationship provides a method to construct the Schur functions via Bernstein operators. Namely, the Bernstein operator
    \begin{equation}\label{eqn: bernstein}
        \mathbf{B}_n=\sum_{i\geq0}(-1)^i h_{n+i}e_i^\perp
    \end{equation}
    is a sum involving both the homogeneous and elementary symmetric functions, where $e_i^\perp$ denotes the adjoint of multiplication by $e_i$ in the Hall inner product. The Bernstein operator $\mathbf{B}_n$ appends a row to the Schur function, i.e., $\mathbf{B}_n(S_\lambda)=S_{(n,\lambda)}$ where $(n,\lambda)=(n,\lambda_1,\lambda_2,\ldots)$. And so, it follows by iteration that $\mathbf{B}_{\lambda_1}\cdots\mathbf{B}_{\lambda_n}(1)=S_\lambda$ \cite[p.~69]{zelevinsky:1981}. There are similar `creation operators' for different bases that add either rows or columns to a Young diagram \cite{zabrocki:2001}.

    Another method of constructing the Schur functions involves vertex operators, which are certain infinite-order differential operators used for the construction of representations of Kac-Moody algebras \cite{kac:1990}. The Schur vertex operator
    \[Y(z)=\exp\left(\sum_{i\geq1}z^ix_i\right)\exp\left(\sum_{i\geq1}-\frac{z^{-i}}{i}\frac{\partial}{\partial x_i}\right)\]
    constructs the generating function for the Schur functions \cite[p.~317]{kac:1990}, i.e., 
    \[Y(z_1)\cdots Y(z_n)(1)=\sum_{\lambda\in\Z^n}S_\lambda z^\lambda.\]
    Since Bernstein operators $\mathbf{B}_n$ are essentially its homogeneous components, $Y(z)=\sum_{n\in\Z}\mathbf{B}_nz^n$, one may construct the Schur vertex operator using Bernstein operators \cite[p.~95]{macdonald:1995}. Hence, Bernstein operators can help provide a more combinatorial interpretation of the vertex operator representation. 

    Meanwhile, the Hall-Littlewood functions $Q_\lambda(X;t)$ provide an important basis of the ring $\Lambda(t)$ of symmetric functions with coefficients in $\Q(t)$. These functions specialize to the Schur functions at $t=0$, and they also have a vertex operator representation \cite{jing:1991b}. There have been several different combinatorial constructions of the Hall-Littlewood vertex operator, and related generalizations \cite{garsia:1992,macdonald:1995,shimozono-zabrocki:2001,zabrocki:2000a,zabrocki:2000b}. It is desirable to find a combinatorial interpretation of this Hall-Littlewood vertex operator via a $t$-analogue of Bernstein operators. There have been several different methods used to construct such an operator. Notably, a deformation of the Bernstein operator of the form
    \[\widetilde{\mathbf{B}}_n=\sum_{i\geq0}t^i\mathbf{B}_{n+i}h_i^\perp\]
    can be used as a method of constructing the Hall-Littlewood functions \cite{berg-bergeron-aliola-serrano-zabrocki:2014}. In this paper, we construct an operator of the form
    \[\sum_{i\geq0}(-1)^i u_{n+i}v_i^\perp\]
    that appends a row to the index of $Q_\lambda(X;t)$, where $u_\lambda,v_\lambda$ are two bases of $\Lambda(t)$ such that $v_\lambda=\omega(u_\lambda)$ under the usual involution $\omega:\Lambda\to\Lambda$. This strengthens the approach given in \cite[p.~236-238]{macdonald:1995}, which does not fully realize this version of a $t$-analogue of the Bernstein operator. 
    
    Namely, we show that the involution $\omega$ plays an essential role in the Jing operator, which carries important combinatorial implications. Indeed, creation operator constructions are often expressed in terms of partition conjugates. For example, the Bernstein operator may be written
    \[\mathbf{B}_n=\sum_{i\geq0}(-1)^i S_{(n+i)}S_{(n)'}^\perp\]
    because the conjugate of a row $\lambda=(n)$ is the column $\lambda'=(1^n)$. Since $\omega(S_\lambda)=S_{\lambda'}$, we essentially show that this operator is perhaps best understood in the form
    \[\mathbf{B}_n=\sum_{i\geq0}(-1)^i S_{(n+i)}\omega(S_{(n)})^\perp.\]
    Hence, it may be more useful to use the involution $\omega$ to generalize similar creation operators to Hall-Littlewood functions or Macdonald polynomials, rather than utilizing conjugates.

    We start by explicitly constructing a basis $B_\lambda(X;t)$ that is the image of the Hall-Littlewood functions under the usual involution $\omega:\Lambda\to\Lambda$, i.e., $B_\lambda=\omega (Q_\lambda)$. Both bases generalize the homogeneous and elementary symmetric functions,
    \begin{align*}
        Q_{(n)}(X;0)=B_{(1^n)}(X;0)&=h_n(X),\\
        B_{(n)}(X;0)=Q_{(1^n)}(X;0)&=e_n(X).
    \end{align*}
    Consequently, these bases share many dual properties, and in particular they both specialize to the Schur functions,
    \begin{align*}
        Q_\lambda(X;0)=S_\lambda(X),\qquad\text{and}\qquad
        B_\lambda(X;0)=S_{\lambda'}(X).
    \end{align*}
    The relationship between $Q_\lambda$ and $B_\lambda$ can be summarized with the following commuting diagram,
    % https://q.uiver.app/#q=WzAsNCxbMCwwLCJRX1xcbGFtYmRhIl0sWzEsMCwiQl9cXGxhbWJkYSJdLFswLDEsIlNfXFxsYW1iZGEiXSxbMSwxLCJTX3tcXGxhbWJkYSd9Il0sWzAsMSwiXFxvbWVnYSIsMCx7InN0eWxlIjp7InRhaWwiOnsibmFtZSI6ImFycm93aGVhZCJ9fX1dLFsyLDMsIlxcb21lZ2EiLDAseyJzdHlsZSI6eyJ0YWlsIjp7Im5hbWUiOiJhcnJvd2hlYWQifX19XSxbMCwyLCJ0PTAiLDJdLFsxLDMsInQ9MCJdXQ==
    \[\begin{tikzcd}
    	{Q_\lambda} & {B_\lambda} \\
    	{S_\lambda} & {S_{\lambda'}}
    	\arrow["\omega", tail reversed, from=1-1, to=1-2]
    	\arrow["{t=0}"', from=1-1, to=2-1]
    	\arrow["{t=0}", from=1-2, to=2-2]
    	\arrow["\omega", tail reversed, from=2-1, to=2-2]
    \end{tikzcd}\]
    
    By direct construction of the Jing operator, we show that the operator
    \[\sum_{i\geq0}(-1)^i q_{n+i}b_i^\perp\]
    is the desired $t$-analogue of the Bernstein operator, where $q_n=Q_{(n)}$ and $b_n=B_{(n)}$. In fact, our method also creates the vertex operator for the $B_\lambda$'s, with a corresponding dual operator
    \[\sum_{i\geq0}(-1)^i b_{n+i}q_i^\perp.\]
    We also show that these methods may be used to create vertex operators of new families of functions.
    
    % Upon specialization to $t=0$, the $Q_\lambda$ vertex operator appends a row to $S_\lambda$, and the $B_\lambda$ vertex operator appends a column to $S_\lambda$.

    Next, one often studies sequences of symmetric functions where the first part of an index is increasing, and vertex operators can be used to prove that these sequences stabilize. For example, the vertex operator method has been used to prove plethysm stability theorems for Schur functions \cite{carre-thibon:1992,scharf-thibon:1994} and Schur's $Q$-functions \cite{graf-jing:2024b}. Now, we can extend these stability methods to Hall-Littlewood functions. We use this method to prove the stability of certain skew structure coefficients $f_{\mu\nu}^\lambda(t)=(Q_{\lambda/\mu},Q_\nu)$, where $Q_{\lambda/\mu}=\sum_\nu f_{\mu\nu}^\lambda(t)Q_\nu$. By setting $t=0$, this implies the stability of the Littlewood-Richardson coefficients.

    Moreover, the coefficient $f_{\mu\nu}^\lambda(t)$ is proportional to the Hall polynomial $g_{\mu\nu}^\lambda(t)$. The Hall polynomial arises in group theory since $g_{\mu\nu}^\lambda(p)$ gives the number of subgroups $B$ of type $\nu$ of a finite abelian $p$-group $G$ of type $\lambda$ such that the quotient group $G/B$ has type $\mu$ \cite{morris:1962}. We show that, as a consequence of the stability of the skew coefficients, the Hall polynomials also stabilize.

    \section{Preliminaries}\label{section: preliminaries}

    \subsection{Compositions and Partitions}

    The many families of symmetric functions are each indexed by integer partitions. However, vertex operator constructions allow one to consider indices that are compositions with negative parts.

    A \emph{composition} is a sequence of integers $\lambda=(\lambda_1,\ldots,\lambda_n)\in\Z^n$. It is a \emph{partition} if its \emph{parts} satisfy $\lambda_1\geq\lambda_2\geq\cdots\geq\lambda_n\geq0$. For two compositions $\lambda\in\Z^m,\mu\in\Z^n$, denote $(\lambda,\mu):=(\lambda_1,\ldots,\lambda_m,\mu_1,\ldots,\mu_n)\in\Z^{m+n}$. We may identify two compositions $\lambda,\mu$ if they differ only by a finite sequence of trailing $0$'s, i.e., if $\mu=(\lambda,0)$. For a partition $\lambda$, let $\lambda'$ denote its \emph{conjugate}, let its \emph{length} $\ell(\lambda)$ be the number of nonzero parts, and let its \emph{weight} $|\lambda|$ be the sum of its parts.

    For $i<j$, we define the \emph{raising operator} $R_{ij}$ to act on a composition $\lambda$ by
    \[R_{ij}\lambda:=(\lambda_1,\ldots,\lambda_i+1,\ldots,\lambda_j-1,\ldots,\lambda_n).\]
    If $u_\lambda$ is a symmetric function indexed by a composition $\lambda$, then a raising operator acts on $u_\lambda$ by $R_{ij}u_\lambda:=u_{R_{ij}\lambda}$.

    \subsection{The ring $\Lambda(t)$}

    We will mainly use the notation of \cite{macdonald:1995}, but with the plethystic notation of \cite{lascoux:2003}. The majority of the results in this chapter can be found in these two sources.
    
    Let $\Lambda$ denote the ring of symmetric functions in the \emph{alphabet} of variables $X=\{x_1,x_2,\ldots\}$, with coefficients in $\Q$. Let $\Lambda(t):=\Lambda\ox\Q(t)$ be the ring of symmetric functions with coefficients in $\Q(t)$, where $t$ is a parameter. 

    For our purposes, it is most convenient to define the Hall-Littlewood functions $Q_\lambda$ via their generating function, and the functions $B_\lambda$ are constructed in an analogous manner. First, define the functions $q_n\in\Lambda(t)$ by the generating function
    \begin{equation}
        \alpha_z:=\prod_{x\in X}\frac{1-txz}{1-xz}=\sum_{n\in\Z}q_n(X;t)z^n.
    \end{equation}
    Define the functions $b_n\in\Lambda(t)$ by
    \begin{equation}
        \beta_z:=\prod_{x\in X}\frac{1+xz}{1+txz}=\sum_{n\in\Z}b_n(X;t)z^n.
    \end{equation}
    And so, we have
    \begin{equation}\label{eqn: alpha_z=beta_-z^-1}
        \alpha_z\beta_{-z}=1,
    \end{equation}
    which generalizes the fundamental identity \eqref{eqn: sigma lambda = 1} in the ring $\Lambda$.

    For any composition $\lambda$, we define the \emph{Hall-Littlewood function} $Q_\lambda(X;t)$ to be the coefficient of $z^\lambda:=z_1^{\lambda_1}z_2^{\lambda_2}\cdots$ in
    \begin{equation}
        \alpha_{z_1,z_2,\ldots}:=\prod_{i\geq1}\alpha_{z_i}\prod_{i<j}\beta_{-z_j}[1/z_i]=\prod_{i\geq1}\alpha_{z_i}\prod_{i<j}\frac{1-z_i^{-1}z_j}{1-tz_i^{-1}z_j},
    \end{equation}
    where $\beta_{-z_j}[1/z_i]=\frac{1-z_i^{-1}z_j}{1-tz_i^{-1}z_j}$ is written using plethystic notation (see section \ref{section: plethysm definitions} for more details on plethysm). We note that this generating function may be constructed with methods independent of those described later in this article \cite[p.~211]{macdonald:1995}. Similarly, define $B_\lambda(X;t)$ to be the coefficient of $z^\lambda$ in
    \begin{equation}
        \beta_{z_1,z_2,\ldots}:=\prod_{i\geq1}\beta_{z_i}\prod_{i<j}\beta_{-z_j}[1/z_i]=\prod_{i\geq1}\beta_{z_i}\prod_{i<j}\frac{1-z_i^{-1}z_j}{1-tz_i^{-1}z_j}.
    \end{equation}
    For any composition $\lambda$, define $q_\lambda:=q_{\lambda_1}q_{\lambda_2}\cdots$ and $b_\lambda:=b_{\lambda_1}b_{\lambda_2}\cdots$. It follows that
    \begin{equation}\label{eqn: raising operator formulas}
        Q_\lambda=\prod_{i<j}\frac{1-R_{ij}}{1-tR_{ij}}q_\lambda,\qquad\text{and}\qquad B_\lambda=\prod_{i<j}\frac{1-R_{ij}}{1-tR_{ij}}b_\lambda.
    \end{equation}

    The families of functions $\{Q_\lambda\}$ and $\{q_\lambda\}$, indexed by partitions $\lambda$, form bases of $\Lambda(t)$, and we will soon show that $\{B_\lambda\}$ and $\{b_\lambda\}$ are bases as well. We note, however, that our definitions of these families of functions are valid for all compositions $\lambda$.

    \subsection{Specializations}

    When $t=0$, the functions $q_n$ and $b_n$ specialize to the homogeneous and elementary symmetric functions, respectively. Define the \emph{homogeneous symmetric functions} $h_n\in\Lambda$ by
    \begin{equation}\label{eqn: sigma_z}
        \sigma_z:=\prod_{x\in X}\frac{1}{1-xz}=\sum_{n\in\Z}h_n(X)z^n
    \end{equation}
    and the \emph{elementary symmetric functions} $e_n\in\Lambda$ by
    \begin{equation}
        \lambda_z:=\prod_{x\in X}(1+xz)=\sum_{n\in\Z}e_n(X)z^n
    \end{equation}
    so that $\sigma_z(X)=\alpha_z(X;0)$ and $\lambda_z(X)=\beta_z(X;0)$. It follows that
    \begin{equation}\label{eqn: sigma lambda = 1}
        \sigma_z\lambda_{-z}=1.
    \end{equation}
    
    For any composition $\lambda\in\Z^n$, we define the Schur function $S_\lambda\in\Lambda$ by the Jacobi-Trudi identity
    \begin{equation*}
        S_\lambda(X):=\det(h_{\lambda_i-i+j})=\det(e_{\lambda_i'-i+j}).
    \end{equation*}
    
    When $t=-1$, the Hall-Littlewood functions $Q_\lambda$ specialize to Schur's $Q$-functions. Define the functions $q_n'\in\Lambda$ by
    \begin{equation*}
        \kappa_z:=\prod_{x\in X}\frac{1+xz}{1-xz}=\sum_{n\in\Z}q_n'(X)z^n
    \end{equation*}
    so that $\kappa_z(X)=\sigma_z(X)\lambda_z(X)=\alpha_z(X;-1)=\beta_z(X;-1)$. For any composition $\lambda\in\Z^{2n}$ (where we may set $\lambda_{2n}=0$), define \emph{Schur's $Q$-function} $Q'_\lambda\in\Lambda$ by
    \begin{equation*}
        Q'_\lambda(X):=\Pf M(\lambda),
    \end{equation*}
    where $M(\lambda)$ is the skew-symmetric matrix with $(i,j)$-entry
    \begin{equation*}
        M(\lambda)_{ij}:=
        \begin{cases}
            q'_{\lambda_i}q'_{\lambda_j}+2\sum_{k=1}^{\lambda_j}(-1)^kq'_{\lambda_i+k}q'_{\lambda_j-k}&\text{if }j>i,\\
            0&\text{if }j=i,\\
            -\left(q'_{\lambda_j}q'_{\lambda_i}+2\sum_{k=1}^{\lambda_i}(-1)^kq'_{\lambda_j+k}q'_{\lambda_i-k}\right)&\text{if }j<i,
        \end{cases}
    \end{equation*}
    and its \emph{Pfaffian} satisfies $\det M(\lambda)=(\Pf M(\lambda))^2$. Schur's $Q$-functions, indexed by strict partitions, form a basis of the subring $\Gamma:=\Q[p_1,p_3,p_5,\ldots]$, where $p_n(X):=\sum_{x\in X}x^n$ is the $n$th \emph{power sum} symmetric function.

    Additionally, for any partition $\lambda$, we define the \emph{monomial symmetric function} $m_\lambda$ by
    \[m_\lambda:=\sum_\mu x^\mu,\]
    where the sum ranges over all distinct permutations $\mu$ of $(\lambda_1,\ldots,\lambda_n,0,0,\ldots)$.

    Let $\omega:\Lambda\to\Lambda$ be the usual involution $\omega(h_n)=e_n$, and extend it to $\Lambda(t)$ by linearity. Then it follows that
    \begin{align*}
        \omega(\alpha_z)&=\omega\left(\sigma_z\lambda_{-tz}\right)
        =\lambda_z\sigma_{-tz}
        =\beta_z,
    \end{align*}
    and so $\omega(q_n)=b_n$. It follows from the raising operator formulas \eqref{eqn: raising operator formulas} that $\omega(Q_\lambda)=B_\lambda$. Since the $q_n$'s are algebraically independent, then the same is true for the $b_n$'s. Similarly, since $\{Q_\lambda\}$ and $\{q_\lambda\}$ form bases of $\Lambda(t)$, each indexed by partitions $\lambda$, then  $\{B_\lambda\}$ and $\{b_\lambda\}$ are bases too. Hence, we have
    \[\Lambda(t)=\Q(t)[q_1,q_2,q_3,\ldots]=\Q(t)[b_1,b_2,b_3,\ldots].\]

    \subsection{Inner Product, Skew Functions, and Adjoints}

    We define an inner product $(\cdot,\cdot)$ on $\Lambda(t)$ by
    \[(q_\lambda(X;t),m_\mu(X))=\delta_{\lambda\mu},\]
    for two partitions $\lambda,\mu$. Since $\omega$ is an isometry, it follows that
    \[(b_\lambda(X;t),f_\mu(X))=\delta_{\lambda\mu},\]
    where $f_\mu:=\omega(m_\mu)$ are the \emph{forgotten symmetric functions}. Additionally, for any partitions $\lambda$ and $\mu$ we have
    \begin{align*}
        (Q_\lambda,Q_\mu)&=c_\lambda(t)\delta_{\lambda\mu},\\
        (B_\lambda,B_\mu)&=c_\lambda(t)\delta_{\lambda\mu},
    \end{align*}
    where $c_\lambda(t):=\prod_{i\geq1}\prod_{j=1}^{k_i}(1-t^j)$ for $\lambda=(1^{k_1},2^{k_2},\ldots)$.

    For any partitions $\lambda$ and $\mu$, we define the skew functions $Q_{\lambda/\mu}$ and $B_{\lambda/\mu}$ by
    \begin{align*}
        (Q_{\lambda/\mu},Q_\nu)&=(Q_\lambda,c_\mu(t)Q_\mu Q_\nu),\\
        (B_{\lambda/\mu},B_\nu)&=(B_\lambda,c_\mu(t)B_\mu B_\nu),
    \end{align*}
    for all partitions $\nu$. It follows that $\omega(Q_{\lambda/\mu})=B_{\lambda/\mu}$.

    For a function $F\in\Lambda(t)$, let $F^\perp$ denote the adjoint of multiplication by $F$ with respect to the inner product $(\cdot,\cdot)$, 
    \begin{align*}
        (F^\perp G,H)=(G,FH),\quad\text{for all } G,H\in\Lambda(t).
    \end{align*}
    For a power series $F=\sum_{n\in\Z}F_nz^n$, denote $F^\perp:=\sum_{n\in\Z}z^nF_n^\perp$.

    It follows that $Q_\mu^\perp Q_\lambda=c_\mu(t)Q_{\lambda/\mu}$ and $B_\mu^\perp B_\lambda=c_\mu(t)B_{\lambda/\mu}$. In particular, we will make use of the identity
    \begin{equation}\label{eqn: q_n^perp Q_lambda}
        q_n^\perp Q_\lambda=
            \begin{cases}
                Q_\lambda&\text{if }n=0,\\
                (1-t)Q_{\lambda/(n)}&\text{if }n>0.
            \end{cases}
    \end{equation}

    \subsection{Plethysm}\label{section: plethysm definitions}

    In the $\lambda$-ring setting (see \cite{lascoux:2003}), plethysm is viewed as the action of a symmetric function on a polynomial in $\C[Y]$, where $Y$ is some alphabet that may contain $X$. In particular, we will often use an alphabet $Y$ containing any of the variables $X,t,z,z_1,z_2,\ldots$. For $P=\sum_\mu c_\mu y^\mu\in\C[Y]$, we define the plethysm $h_n[P]$ by the generating function
    \begin{equation*}
        \sigma_z[P]:=\prod_{\mu}\left(\frac{1}{1-zy^\mu}\right)^{c_\mu}=\sum_{n\in\Z}h_n[P]z^n.
    \end{equation*}
    We can write any $F\in\Lambda(t)$ as a polynomial in the $h_n$'s, say $F(A;t)=\mathcal{F}(h_1,h_2,\ldots)$. Then we define the plethysm $F[P]:=\mathcal{F}(h_1[P],h_2[P],\ldots)$. Note that we will not specialize $t$ when using plethystic notation, so we will write $F[X]$ to mean $F(X;t)$. 

    It follows from the identity $\sigma_z\lambda_{-z}=1$ that we can compute $e_n[P]$ by 
    \begin{equation*}
        \lambda_z[P]:=\prod_{\mu}(1+zy^\mu)^{c_\mu}=\sum_{n\in\Z}e_n[P]z^n.
    \end{equation*}
    Since $\alpha_z=\sigma_z\lambda_{-tz}$, we have
    \begin{equation}
        \alpha_z[P]:=\prod_{\mu}\left(\frac{1-tyz}{1-yz}\right)^{c_\mu}=\sum_{n\in\Z}q_n[P]z^n.
    \end{equation}
    Similarly, we get
    \begin{equation}
        \beta_z[P]:=\prod_{\mu}\left(\frac{1+yz}{1+tyz}\right)^{c_\mu}=\sum_{n\in\Z}b_n[P]z^n.
    \end{equation}
    It follows that $Q_\lambda[P]$ is the coefficient of $z^\lambda$ in
    \begin{equation*}
        \alpha_{z_1,z_2,\ldots}[P]:=\prod_{i\geq1}\alpha_{z_i}[P]\prod_{i<j}\beta_{-z_j}[1/z_i],
    \end{equation*}
    and similarly $B_\lambda[P]$ is the coefficient of $z^\lambda$ in $\beta_{z_1,z_2,\ldots}[P]$.

    Since $F[X]=F[x_1+x_2+\cdots]$, we identify an alphabet with the sum of its elements. The sum $X+Y$ of two alphabets is defined to be the \emph{disjoint} union of $X$ and $Y$. It follows that $kX=\underbrace{X+\cdots+X}_{k\text{ times}}$ for all integers $k\geq0$, and in particular
    \[\sigma_z[X+Y]=\sigma_z[X]\sigma_z[Y].\]
    We extend this property to all $k\in\C$ via the identity $\sigma_z[kX]=(\sigma_z[X])^k$, and so $\sigma_z[-X]=(\sigma_z[X])^{-1}=\lambda_{-z}[X]$. Therefore, we have
    \begin{equation*}
        \alpha_z[kX]=(\alpha_z[X])^k,\qquad\text{and}\qquad\beta_z[kX]=(\beta_z[X])^k,
    \end{equation*}
    for all $k\in\C$, and hence
    \begin{align*}
        \alpha_z[-X]=\beta_{-z}[X],\qquad\text{and}\qquad\beta_z[-X]=\alpha_{-z}[X].
    \end{align*}
    It follows that $q_n[-X]=(-1)^nb_n[X]$. Thus, if $F\in\Lambda(t)$ is homogeneous of degree $n$, then we have
    \begin{equation*}
        F[-X]=(-1)^n(\omega F)[X].
    \end{equation*}
    Moreover, we have the useful identities \cite[p.~228]{macdonald:1995}
    \begin{align*}
        Q_\lambda[X+Y]&=\sum_{\mu}Q_{\lambda/\mu}[X]Q_\mu[Y],\\
        B_\lambda[X+Y]&=\sum_{\mu}B_{\lambda/\mu}[X]B_\mu[Y].
    \end{align*}

    \section{Properties of $B_\lambda$}

    Although the Hall-Littlewood functions $Q_\lambda(X;t)$ are widely studied, the functions $B_\lambda(X;t)$ are not. Hence, in this section we develop some properties of this basis that will be useful in later sections.

    \subsection{$B_\lambda$ Identities}
    
    First, it is well-known that setting $t=0$ specializes the Hall-Littlewood functions to the Schur functions, and setting $t=-1$ results in Schur's $Q$-functions. Using the involution $\omega$, we see that $B_\lambda$ specializes as follows.
    
    \begin{proposition}
        We have $Q_\lambda(X;0)=B_{\lambda'}(X;0)$, and in particular
        \begin{align*}
            Q_\lambda(X;0)&=S_\lambda(X),&
            B_\lambda(X;0)&=S_{\lambda'}(X),\\
            Q_\lambda(X;-1)&=Q'_\lambda(X),&
            B_\lambda(X;-1)&=Q'_\lambda(X).
        \end{align*}
    \end{proposition}
    
    \begin{proof}
        It is well-known the $Q_\lambda(X;0)$ specializes to $S_\lambda$. For the second identity, we have
        \begin{align*}
            B_\lambda(X;0)&=\omega Q_\lambda(X;0)\\
            &=\omega S_\lambda(X)\\
            &=S_{\lambda'}(X).
        \end{align*}
        Lastly, $Q'_\lambda(X)=Q_\lambda(X;-1)$ is well-known, and we get $Q'_\lambda(X)=B_\lambda(X;-1)$ by applying the involution $\omega$, since $\omega$ acts as the identity on the subring $\Gamma\subset\Lambda$.
    \end{proof}

    In particular, we have the following dual property of $Q_\lambda$ and $B_\lambda$.

    \begin{corollary}
        We have
        \begin{align*}
            Q_{(n)}(X;0)=B_{(1^n)}(X;0)&=h_n(X),\\
            B_{(n)}(X;0)=Q_{(1^n)}(X;0)&=e_n(X).
        \end{align*}
    \end{corollary}
    
    \begin{proof}
        This follows since $\alpha_z(X;0)=\sigma_z(X)$ and $\beta_z(X;0)=\lambda_z(X)$.
    \end{proof}

    To proceed further, we will need the decompositions of $q_n$ and $b_n$ into the monomial symmetric functions.

    \begin{proposition}\label{b_n monomials}
        For all $n\in\Z$, we have
        \begin{align}
            q_n&=\sum_{|\lambda|=n}(1-t)^{\ell(\lambda)}m_\lambda,\\
            b_n&=\sum_{|\lambda|=n}(-t)^{n-\ell(\lambda)}(1-t)^{\ell(\lambda)}m_\lambda,
        \end{align}
        where the sums range over partitions $\lambda$.
    \end{proposition}

    \begin{proof}
        First, we expand out
        \begin{align*}
            \alpha_z&=\prod_{x\in X}\frac{1-txz}{1-xz}\\
            &=\prod_{x\in X}(1-txz)\sum_{n\geq0}(xz)^n\\
            &=\prod_{x\in X}\left(\sum_{n\geq0}(xz)^n-t\sum_{n\geq1}(xz)^n\right)\\
            &=\prod_{x\in X}\left(1+(1-t)\sum_{n\geq1}x^nz^n\right).
        \end{align*}
        Note that from (\ref{eqn: sigma_z}) that we have
        \begin{align*}
            \sigma_z=\prod_{x\in X}(1-zx)^{-1}=\prod_{x\in X}\left(1+\sum_{n\geq1}x^nz^n\right),
        \end{align*}
        and recall that $m_\lambda=\sum_{\mu} x^{\mu}$, where the sum ranges over distinct permutations of $\lambda$. Hence, we have that the coefficient of $z^n$ in $\sigma_z$ is $h_n=\sum_{|\lambda|=n}m_\lambda$ \cite[p.~21]{macdonald:1995}. So, compared to $\sigma_z$, the coefficient of $z^n$ in $\alpha_z$ has an additional factor of $(1-t)^{\ell(\lambda)}$ in each term.

        Similarly, we expand $\beta_z$ to get
        \begin{align*}
            \beta_z&=\prod_{x\in X}\frac{1+xz}{1+txz}\\
            &=\prod_{x\in X}(1+xz)\sum_{n\geq0}(-txz)^n\\
            &=\prod_{x\in X}\left(\sum_{n\geq0}(-txz)^n+\sum_{n\geq1}(-t)^{n-1}(xz)^n\right).
        \end{align*}
        Then, we can regroup terms to get
        \begin{align*}
            \beta_z&=\prod_{x\in X}\left(1+\sum_{n\geq1}\left((-tx)^n+(-t)^{n-1}x^n\right)z^n\right)\\
            &=\prod_{x\in X}\left(1+(1-t)\sum_{n\geq1}(-t)^{n-1}x^nz^n\right).
        \end{align*}
        Now, we can see that compared to $\alpha_z$, each term in the coefficient to $z^n$ has an additional factor of $(-t)^{n-\ell(\lambda)}$.
    \end{proof}

    \subsection{Inequivalence of $Q_{\lambda'}$ and $B_\lambda$}

    The relationship between $Q_\lambda$ and $B_\lambda$ can be described with the following commuting diagram,
    % https://q.uiver.app/#q=WzAsNCxbMCwwLCJRX1xcbGFtYmRhIl0sWzEsMCwiQl9cXGxhbWJkYSJdLFswLDEsIlNfXFxsYW1iZGEiXSxbMSwxLCJTX3tcXGxhbWJkYSd9Il0sWzAsMSwiXFxvbWVnYSIsMCx7InN0eWxlIjp7InRhaWwiOnsibmFtZSI6ImFycm93aGVhZCJ9fX1dLFsyLDMsIlxcb21lZ2EiLDAseyJzdHlsZSI6eyJ0YWlsIjp7Im5hbWUiOiJhcnJvd2hlYWQifX19XSxbMCwyLCJ0PTAiLDJdLFsxLDMsInQ9MCJdXQ==
    \[\begin{tikzcd}
    	{Q_\lambda} & {B_\lambda} \\
    	{S_\lambda} & {S_{\lambda'}}
    	\arrow["\omega", tail reversed, from=1-1, to=1-2]
    	\arrow["{t=0}"', from=1-1, to=2-1]
    	\arrow["{t=0}", from=1-2, to=2-2]
    	\arrow["\omega", tail reversed, from=2-1, to=2-2]
    \end{tikzcd}\]
    Since $\omega$ acts on the Schur functions by conjugating the index, i.e., $\omega(S_\lambda)=S_{\lambda'}$, it is natural to ask if this is also the case with the Hall-Littlewood functions. In other words, is $Q_{\lambda'}$ equal to $B_\lambda$? A simple example shows that these are not equal in general.

    \begin{example}
        Consider the partition $\lambda=(2)$, and let $X=\{x_1,x_2\}$. Using our definitions and computer algebra, we get
        \begin{align*}
            Q_{(1^2)}(X;t)&=\left(t^{3}-t^{2}-t+1\right)x_{1}x_{2},\\
            B_{(2)}(X;t)&=\left(t^{2}-t\right)x_{1}^{2}+\left(t^{2}-2\,t+1\right)x_{1}x_{2}+\left(t^{2}-t\right)x_{2}^{2},
        \end{align*}
        and hence $Q_{(1^2)}\neq B_{(2)}$. Similarly, we have
        \begin{align*}
            Q_{(2)}(X;t)&=\left(-t+1\right)x_{1}^{2}+\left(t^{2}-2\,t+1\right)x_{1}x_{2}+\left(-t+1\right)x_{2}^{2},\\
            B_{(1^2)}(X;t)&=\left(t^{3}-t^{2}-t+1\right)x_{1}^{2}+\left(t^{3}-t^{2}-t+1\right)x_{1}x_{2}+\left(t^{3}-t^{2}-t+1\right)x_{2}^{2}.
        \end{align*}
        Indeed, it is also clear that
        \begin{align*}
            Q_{(1^2)}&=(t^{3}-t^{2}-t+1)e_2,\\
            B_{(1^2)}&=(t^{3}-t^{2}-t+1)h_2
        \end{align*}
        are mapped to each other under the involution $\omega$. Similarly, we have 
        \begin{align*}
            Q_{(2)}&=(1-t)h_2+(t^2-t)e_2,\\
            B_{(2)}&=(1-t)e_2+(t^2-t)h_2.
        \end{align*}
    \end{example}

    \section{Vertex Operator Identity}

    Vertex operators can be used to construct the Schur functions \cite[p.~317]{kac:1990}, Schur's $Q$-functions \cite{jing:1991a}, and the Hall-Littlewood functions \cite{jing:1991b}. We can use these constructions to get useful identities in the language of symmetric functions. In particular, Schur functions and Schur's $Q$-functions have been shown via their determinantal formulas \cite{carre-thibon:1992,graf-jing:2024a} to satisfy the following vertex operator identities,
    \begin{align*}
        \sigma_z\lambda_{-1/z}^\perp S_\lambda&=\sum_{n\in\Z}S_{(n,\lambda)}z^n,\\
        \kappa_z\kappa_{-1/z}^\perp Q'_\lambda&=\sum_{n\in\Z}Q'_{(n,\lambda)}z^n,
    \end{align*}
    where $(n,\lambda):=(n,\lambda_1,\lambda_2,\ldots)$. It would be useful to have an analogous determinantal proof for Hall-Littlewood functions, but there is no suitable determinantal formula for these functions. Therefore, we will prove the analogous vertex operator identity for Hall-Littlewood functions with a generating function method.

    \subsection{Hall-Littlewood Vertex Operator}
    
    % First, we introduce some necessary identities.

    % \begin{proposition}[Sum Rule]
    %     We have
    %     \begin{align*}%p228
    %         Q_\lambda[X+Y]&=\sum_\mu Q_{\lambda/\mu}[X]Q_\mu[Y],\label{eqn: sum rule}\\
    %         B_\lambda[X+Y]&=\sum_\mu B_{\lambda/\mu}[X]B_\mu[Y].
    %     \end{align*}
    % \end{proposition}

    % In particular, when an alphabet has only one variable $z$, then $Q_\lambda[z]=0$ for $\ell(\lambda)>1$. So in that case, the sum rule reduces to
    % \begin{align*}
    %     Q_\lambda[X+z]&=\sum_{i\geq0} Q_{\lambda/(i)}[X]q_i[z].
    % \end{align*}
        
    % \begin{lemma}[{\cite[p.~236]{macdonald:1995}}]\label{(q_n,GH)}
    %     Let $n\in\Z$ and $G,H\in\Lambda(t)$. Then we have
    %     \begin{equation*}
    %         (q_n,GH)=\sum_{r+s=n}(q_r,G)(q_s,H).
    %     \end{equation*}
    % \end{lemma}

    We proceed in a manner similar to \cite[p.~236-238]{macdonald:1995}. First, the following proposition provides useful formulas for computing the action of the adjoint.

    \begin{proposition}
        We have
        \begin{align*}
            q_k^\perp q_n&=
                \begin{cases}
                    q_n&\text{if }k=0,\\
                    (1-t)q_{n-k}&\text{if }k>0,
                \end{cases}&
            b_k^\perp q_n&=
                \begin{cases}
                    q_n&\text{if }k=0,\\
                    (-t)^{k-1}(1-t)q_{n-k}&\text{if }k>0,
                \end{cases}\\
            q_k^\perp b_n&=
                \begin{cases}
                    b_n&\text{if }k=0,\\
                    (-t)^{k-1}(1-t)b_{n-k}&\text{if }k>0,
                \end{cases}&
            b_k^\perp b_n&=
                \begin{cases}
                    b_n&\text{if }k=0,\\
                    (1-t)b_{n-k}&\text{if }k>0.
                \end{cases}
        \end{align*}
        % \begin{align}
        %     q_k^\perp q_n&=
        %         \begin{cases}
        %             q_n&\text{if }k=0,\\
        %             (1-t)q_{n-k}&\text{if }k>0.
        %         \end{cases}\\
        %     q_k^\perp b_n&=
        %         \begin{cases}
        %             b_n&\text{if }k=0,\\
        %             (-t)^{k-1}(1-t)b_{n-k}&\text{if }k>0.
        %         \end{cases}\\
        %     b_k^\perp q_n&=
        %         \begin{cases}
        %             q_n&\text{if }k=0,\\
        %             (-t)^{k-1}(1-t)q_{n-k}&\text{if }k>0.
        %         \end{cases}\\
        %     b_k^\perp b_n&=
        %         \begin{cases}
        %             b_n&\text{if }k=0,\\
        %             (1-t)b_{n-k}&\text{if }k>0.
        %         \end{cases}
        % \end{align}
    \end{proposition}

    \begin{proof}
        The first identity is an immediate consequence of \eqref{eqn: q_n^perp Q_lambda}. To compute $b_k^\perp q_n$, assume $k>0$. By the definition of $b_k^\perp$, we have
        \begin{align*}
            (b_k^\perp q_n,H)=(q_n,b_k H)
        \end{align*}
        for all $H\in\Lambda(t)$. Note that $(q_n,GH)=\sum_{r+s=n}(q_r,G)(q_s,H)$ for all $G,H$ \cite[p.~236]{macdonald:1995}, and so we have
        \begin{align*}
            (q_n,b_k H)=\sum_{r+s=n}(q_r,b_k)(q_s,H).
        \end{align*}
        From Proposition \ref{b_n monomials}, this is
        \begin{align*}
            (q_n,b_k H)=\sum_{r+s=n}\left(q_r,\sum_{|\lambda|=k}(-t)^{k-\ell(\lambda)}(1-t)^{\ell(\lambda)}m_\lambda\right)(q_s,H).
        \end{align*}
        Since the $q_\lambda$ and $m_\lambda$ are dual, we have that the first inner product is zero unless $\lambda=(k)=(r)$, and so $s=n-k$. Thus, we are left with
        \begin{align*}
            (q_n,b_k H)=(q_k,(-t)^{k-1}(1-t)^1m_k)(q_{n-k},H).
        \end{align*}
        Now, we can pull out coefficients in the first inner product, and use the fact that $(q_k,m_k)=1$, and so we get
        \begin{align*}
            (q_n,b_k H)=(-t)^{k-1}(1-t)(q_{n-k},H)=((-t)^{k-1}(1-t)q_{n-k},H).
        \end{align*}
        Thus, we have
        \begin{align*}
            (b_k^\perp q_n,H)=((-t)^{k-1}(1-t)q_{n-k},H).
        \end{align*}
        Since this is true for all $H\in\Lambda(t)$, we must have that $b_k^\perp q_n=(-t)^{k-1}(1-t)q_{n-k}$. Finally, we get the other two identities by applying $\omega$.
    \end{proof}

    Next, we compute the plethysm $q_m[z]$ and $b_m[z]$.

    \begin{proposition}
        We have
        \begin{align*}
            q_m[z]&=
                \begin{cases}
                    1&\text{if }m=0,\\
                    (1-t)z^m&\text{if }m>0,
                \end{cases}&
            b_m[z]&=
                \begin{cases}
                    1&\text{if }m=0,\\
                    (-t)^{m-1}(1-t)z^m&\text{if }m>0.
                \end{cases}
        \end{align*}
    \end{proposition}

    \begin{proof}
        This is a straightforward expansion of $\alpha_w[z]$ and $\beta_w[z]$, where $w$ is another indeterminate.
    \end{proof}

    Together, these previous propositions can be used to compute the actions of $\alpha_{\pm z}^\perp$ and $\beta_{\pm z}^\perp$ on $\alpha_w$ and $\beta_w$.
    
    \begin{proposition}\label{prop: action of adjoint}
        We have
        % \begin{align*}
        %     \alpha_z^\perp \alpha_w[X]&=\alpha_w[X]\alpha_w[z]=\alpha_w[X+z],\\
        %     \alpha_z^\perp \beta_w[X]&=\beta_w[X]\beta_w[z]=\beta_w[X+z],\\
        %     \beta_{-z}^\perp \alpha_w[X]&=\alpha_w[X]\beta_{-w}[z]=\alpha_w[X-z],\\
        %     \beta_{-z}^\perp \beta_w[X]&=\beta_w[X]\alpha_{-w}[z]=\beta_w[X-z],\\
        %     \alpha_{-z}^\perp \alpha_w[X]&=\alpha_w[X]\alpha_{-w}[z],\\
        %     \alpha_{-z}^\perp \beta_w[X]&=\beta_w[X]\beta_{-w}[z],\\
        %     \beta_{z}^\perp \alpha_w[X]&=\alpha_w[X]\beta_w[z],\\
        %     \beta_{z}^\perp \beta_w[X]&=\beta_w[X]\alpha_w[z].
        % \end{align*}
        \begin{align*}
            \alpha_z^\perp(\alpha_w[X])&=\alpha_w[X]\alpha_w[z]=\alpha_w[X+z],&
            \beta_{z}^\perp (\alpha_w[X])&=\alpha_w[X]\beta_w[z],\\
            \alpha_z^\perp (\beta_w[X])&=\beta_w[X]\beta_w[z]=\beta_w[X+z],&
            \beta_{z}^\perp (\beta_w[X])&=\beta_w[X]\alpha_w[z],\\
            \alpha_{-z}^\perp (\alpha_w[X])&=\alpha_w[X]\alpha_{-w}[z],&
            \beta_{-z}^\perp (\alpha_w[X])&=\alpha_w[X]\beta_{-w}[z]=\alpha_w[X-z],\\
            \alpha_{-z}^\perp (\beta_w[X])&=\beta_w[X]\beta_{-w}[z],&
            \beta_{-z}^\perp (\beta_w[X])&=\beta_w[X]\alpha_{-w}[z]=\beta_w[X-z].
        \end{align*}
        In other words, $\alpha_w[X]$ and $\beta_w[X]$ are eigenvectors with respect to the operators $\alpha_{\pm z}^\perp$ and $\beta_{\pm z}^\perp$.
    \end{proposition}

    \begin{proof}
        We compute
        \begin{align*}
            \alpha_z^\perp \alpha_w[X]&=\sum_{m\in\Z}z^mq_m^\perp\sum_{n\in\Z}q_n[X]w^n\\
            &=\sum_{n\in\Z}w^n\sum_{m\in\Z}z^m q_m^\perp q_n[X]\\
            &=\sum_{n\in\Z}w^n\left(1+\sum_{m\geq1}z^m\cdot(1-t)q_{n-m}[X]\right)\\
            &=\sum_{n\in\Z}w^n\sum_{m\geq0}q_{n-m}[X]q_m[z]\\
            &=\sum_{i\in\Z}q_i[X]w^i\cdot \sum_{j\in\Z}q_j[z]w^j\\
            &=\alpha_w[X]\alpha_w[z].
        \end{align*}
        \begin{align*}
            \alpha_z^\perp \beta_w[X]&=\sum_{m\in\Z}z^mq_m^\perp\sum_{n\in\Z}b_n[X]w^n\\
            &=\sum_{n\in\Z}w^n\sum_{m\in\Z}z^m q_m^\perp b_n[X]\\
            &=\sum_{n\in\Z}w^n\left(1+\sum_{m\geq1}z^m\cdot(-t)^{m-1}(1-t)b_{n-m}[X]\right)\\
            &=\sum_{n\in\Z}w^n\sum_{m\geq0}b_{n-m}[X]b_m[z]\\
            &=\sum_{i\in\Z}b_i[X]w^i\cdot \sum_{j\in\Z}b_j[z]w^j\\
            &=\beta_w[X]\beta_w[z].
        \end{align*}
        The rest may be computed similarly, or may be obtained from the first two via the involution $\omega$ and the identities $\alpha_z\beta_{-z}=1$, $\alpha_z^\perp\beta_{-z}^\perp=1$, $\alpha_w[-z]=\beta_{-w}[z]$, and $\alpha_{-w}[z]=\beta_w[-z]$.
    \end{proof}

    Now, note that $\alpha_z^\perp$ is a ring homomorphism \cite[p.~236-237]{macdonald:1995}, and so $\beta_z^\perp$ is as well by an analogous argument. Consequently, if we view $\alpha_w$ and $\beta_w$ as multiplication operators, then we get the following commutation identities.
    \begin{align*}
            \alpha_z^\perp \alpha_w[X]&=\alpha_w[X]\alpha_w[z]\alpha_z^\perp,&
            \beta_{z}^\perp \alpha_w[X]&=\alpha_w[X]\beta_w[z]\beta_{z}^\perp,\\
            \alpha_z^\perp \beta_w[X]&=\beta_w[X]\beta_w[z]\alpha_z^\perp,&
            \beta_{z}^\perp \beta_w[X]&=\beta_w[X]\alpha_w[z]\beta_{z}^\perp,\\
            \alpha_{-z}^\perp \alpha_w[X]&=\alpha_w[X]\alpha_{-w}[z]\alpha_{-z}^\perp,&
            \beta_{-z}^\perp \alpha_w[X]&=\alpha_w[X]\beta_{-w}[z]\beta_{-z}^\perp,\\
            \alpha_{-z}^\perp \beta_w[X]&=\beta_w[X]\beta_{-w}[z]\alpha_{-z}^\perp,&
            \beta_{-z}^\perp \beta_w[X]&=\beta_w[X]\alpha_{-w}[z]\beta_{-z}^\perp.
        \end{align*}

    Thus, as an immediate consequence of Proposition \ref{prop: action of adjoint}, we have a very simple proof of actions of the vertex operator $\alpha_z\beta_{-1/z}^\perp$ and its dual $\beta_z\alpha_{-1/z}^\perp$.
    
    \begin{theorem}\label{thm: vertex operators}
        Let $\lambda$ be a composition, then
        \begin{align}
            \alpha_z\beta_{-1/z}^\perp Q_\lambda&=\sum_{n\in\Z}Q_{(n,\lambda)}z^n,\label{eqn: vertex op action1}\\
            \beta_z\alpha_{-1/z}^\perp B_\lambda&=\sum_{n\in\Z}B_{(n,\lambda)}z^n.\label{eqn: vertex op action2}
        \end{align}
    \end{theorem}

    \begin{proof}
        % We wish to show that the coefficient of $z^nz^\lambda$ in $\alpha_z\beta_{-1/z}^\perp\alpha_{z_1,z_2,\ldots}$ is $Q_{n\lambda}$, or equivalently that
        % \[\alpha_z\beta_{-1/z}^\perp \alpha_{z_1,z_2,\ldots} =\alpha_z\prod_{i\geq1}\alpha_{z_i}\beta_{-z_i}[1/z]\prod_{i<j}\beta_{-z_j}[1/z_i].\]
        % Applying the operator $\alpha_z^\perp$, we see that it suffices to show
        % \[\beta_{-1/z}^\perp \alpha_{z_1,z_2,\ldots}[X] =\prod_{i\geq1}\alpha_{z_i}\beta_{-z_i}[1/z]\prod_{i<j}\beta_{-z_j}[1/z_i].\]
        By iteration on the number of variables $w_1,w_2,\ldots$, we compute
        \begin{align*}
            \alpha_z[X]\beta_{-1/z}^\perp \alpha_{w_1,w_2,\ldots}[X]
            &=\alpha_z[X]\prod_{i\geq1}\beta_{-1/z}^\perp\alpha_{w_i}[X]\prod_{i<j}\beta_{-w_j}[1/w_i]\\
            % &=\alpha_z[X]\prod_{i\geq1}\alpha_{z_i}[X](\alpha_{z_i}[1/z])^{-1}\prod_{i<j}\beta_{-z_j}[1/z_i]\\
            &=\alpha_z[X]\prod_{i\geq1}\alpha_{w_i}[X]\beta_{-w_i}[1/z]\prod_{i<j}\beta_{-w_j}[1/w_i]\\
            &=\alpha_{z,w_1,w_2,\ldots}[X].
        \end{align*}
        Equating the coefficient of $w^\lambda$ at the beginning and ending of the chain of equalities gives the first result. Similarly, we have
        \begin{align*}
            \beta_z[X]\alpha_{-1/z}^\perp \beta_{w_1,w_2,\ldots}[X]
            &=\beta_z[X]\prod_{i\geq1}\alpha_{-1/z}^\perp\beta_{w_i}[X]\prod_{i<j}\beta_{-w_j}[1/w_i]\\
            &=\beta_z[X]\prod_{i\geq1}\beta_{w_i}[X]\beta_{-w_i}[1/z]\prod_{i<j}\beta_{-w_j}[1/w_i]\\
            &=\beta_{z,w_1,w_2,\ldots}[X].
        \end{align*}
        % The operator $\beta_{-1/z}^\perp$ acts via the plethysm $\beta_{-1/z}^\perp F[X]=F[X-1/z]$ \eqref{}, and ``right plethysm by $g$'' defined by $f\mapsto f[g]$ is an algebra homomorphism \cite{loehr-remmel}. Hence the operator $\beta_{-1/z}^\perp$ is an algebra homomorphism, and so we only need to show that $\beta_{-1/z}^\perp \alpha_{z_i}=\alpha_{z_i}\beta_{-z_i}[1/z]$ (noting that $\beta_{-1/z}^\perp\beta_{-z_j}[1/z_i]=\beta_{-z_j}[1/z_i]$). Finally, we compute
        % \begin{align*}
        %     \beta_{-1/z}^\perp\alpha_{z_i}[X]&=\alpha_{z_i}[X-1/z]\\
        %     &=\alpha_{z_i}[X]\alpha_{z_i}^{-1}[1/z]\\
        %     &=\alpha_{z_i}[X]\beta_{-z_i}[1/z].
        % \end{align*}
    \end{proof}
    
     % Our approach differs from the one in \cite[p.~236--238]{macdonald:1995}. In particular, note that Macdonald defines the generating function $\beta_z$ differently, and hence he gets a different identity. Our approach here is preferable since the $b_n$'s specialize the the $e_n$'s.

    Now, let $H(z):=\alpha_z\beta_{-1/z}^\perp$, and define its homogeneous components $H_n$ by $H(z)=\sum_{n\in\Z}H_nz^n$. By \eqref{eqn: vertex op action1}, it follows that
    \[\left(\prod_{i\geq1}H(z_i)\right)(Q_\lambda)=\sum_{\mu}Q_{(\mu,\lambda)}z^\mu.\]
    Since $Q_{(0)}=1$ and $Q_{(\mu,0)}=Q_\mu$, we set $\lambda=0$ to get
    \[\left(\prod_{i\geq1}H(z_i)\right)(1)=\sum_{\mu}Q_{\mu}z^\mu=\alpha_{z_1,z_2,\ldots}.\]
    By equating the coefficients of $z^\mu$ for any composition $\mu\in\Z^n$, we get the usual presentation of the vertex operator \cite{jing:1991b},
    \[H_{\mu_1}\cdots H_{\mu_n}(1)=Q_\mu.\]
    Similarly, let $\overline{H}(z):=\beta_z\alpha_{-1/z}^\perp=\sum_{n\in\Z}\overline{H}_nz^n$, then
    \[\overline{H}_{\mu_1}\cdots \overline{H}_{\mu_n}(1)=B_\mu.\]
    We may equivalently write these equations as $t$-analogues of Bernstein operators,
    \begin{align*}
        \sum_{i\geq0}(-1)^iq_{n+i}b_i^\perp(Q_\lambda)&=Q_{(n,\lambda)},\\
        \sum_{i\geq0}(-1)^ib_{n+i}q_i^\perp(B_\lambda)&=B_{(n,\lambda)}.
    \end{align*}
    Specializing the operators with $t=0$, we recover the identities
    \begin{align*}
        \sum_{i\geq0}(-1)^ih_{n+i}e_i^\perp(S_\lambda)&=S_{(n,\lambda)},\\
        \sum_{i\geq0}(-1)^ie_{n+i}h_i^\perp(S_{\lambda'})&=S_{(n,\lambda)'}.
    \end{align*}
    Note that when $\lambda$ is a partition and $n\geq\lambda_1$ we have $(n,\lambda)'=1^n+\lambda'$, where the sum of two partitions is computed part-wise. Thus, we get the identity
    \begin{align*}
        \sum_{i\geq0}(-1)^ie_{n+i}h_i^\perp(S_{\lambda})&=S_{1^n+\lambda},
    \end{align*}
    for $n\geq\ell(\lambda)$. That is, a column of length $n$ has been appended to $\lambda$.
    
    Note that our indexing differs from both the original notation in \cite{jing:1991b} and from more modern conventions \cite{jing-liu:2022}. Most importantly, our dual operator $\overline{H}(z)$ is a slightly different operator than the dual Jing operator $H^*(z)$. To see this, first let $\Psi_z$ be the generating function of the power sum symmetric functions,
    \[\Psi_z:=\sum_{n\geq1}p_nz^{n-1}.\]
    Since we have
    \begin{align*}
        \Psi_z=\frac{\partial}{\partial z}\log \sigma_z,\qquad\text{and}\qquad \Psi_{-z}=\frac{\partial}{\partial z}\log \lambda_z,
    \end{align*}
    it follows that
    \[\sigma_z=\exp\left(\sum_{n\geq1}\frac{p_nz^n}{n}\right),\qquad\text{and}\qquad\lambda_z=\exp\left(-\sum_{n\geq1}\frac{p_n(-z)^n}{n}\right).\]
    Therefore, since $\alpha_z=\sigma_z/\sigma_{tz}$ and $\beta_z=\lambda_z/\lambda_{tz}$, we have
    \[\alpha_z=\exp\left(\sum_{n\geq1}\frac{1-t^n}{n}p_nz^n\right),\qquad\text{and}\qquad \beta_z=\exp\left(-\sum_{n\geq1}\frac{1-t^n}{n}p_n(-z)^n\right).\]
    It follows that
    \[\alpha_{-1/z}^\perp=\exp\left(\sum_{n\geq1}\frac{1-t^n}{n}p_n^\perp (-z)^{-n}\right),\qquad\text{and}\qquad \beta_{-1/z}^\perp=\exp\left(-\sum_{n\geq1}\frac{1-t^n}{n}p_n^\perp z^{-n}\right).\]
    And so, we have 
    \begin{align*}
        H(z)&=\exp\left(\sum_{n\geq1}\frac{1-t^n}{n}p_nz^n\right)\exp\left(-\sum_{n\geq1}\frac{1-t^n}{n}p_n^\perp z^{-n}\right),\\
        \overline{H}(z)&=\exp\left(-\sum_{n\geq1}\frac{1-t^n}{n}p_n(-z)^n\right) \exp\left(\sum_{n\geq1}\frac{1-t^n}{n}p_n^\perp (-z)^{-n}\right).
    \end{align*}
    It is easy to see that $H(z)$ is the same vertex operator as in \cite{jing-liu:2022}, whereas the dual Jing operator $H^*(z)$ satisfies $H^*(z)=\overline{H}(-z)$.

    % Thus, the homogeneous components of the dual operators satisfy $H^*_n=(-1)^n\overline{H}^*_n$, and so we see that the $H^*_n$ satisfy similar relations as the $\overline{H}^*_m$.

    % \begin{proposition}[\cite{jing:1991b}]
    %     The operators $H_n$ and $H_n^*$ satisfy
    %     \begin{align*}
    %         H_mH_n-tH_nH_m&=tH_{m+1}H_{n-1}-H_{n-1}H_{m+1},\\
    %         H_m^* H_n^*-tH_n^* H_m^*&=tH_{m-1}^* H_{n+1}^*-H_{n+1}^* H_{m-1}^*,\\
    %         H_m H_n^*-tH_n^* H_m&=-tH_{m-1}H_{n-1}^*+H_{n-1}^*H_{m-1}+(-1)^n(1-t)^2\delta_{m,n}.
    %     \end{align*}
    % \end{proposition}

    \subsection{Other Vertex Operators}

    Furthermore, it is clear from Proposition \ref{prop: action of adjoint} and the proof of Theorem \ref{thm: vertex operators} that we may construct many more families of functions via operators similar to $\alpha_z\beta_{-1/z}^\perp$ and $\beta_z\alpha_{-1/z}^\perp$. In other words, these operators build up their corresponding generating functions. For example, the operator $\beta_z\beta_{-1/z}^\perp$ builds the generating function
    \[\prod_{i\geq1}\beta_{z_i}\prod_{i<j}\alpha_{-z_j}[1/z_i],\]
    and the operator $\alpha_z\alpha_{-1/z}^\perp$ builds the generating function
    \[\prod_{i\geq1}\alpha_{z_i}\prod_{i<j}\alpha_{-z_j}[1/z_i].\]
    
    Such functions may merit further study, as it remains to be seen if they exhibit similar useful properties as the Hall-Littlewood functions.

    \section{Stability Theorems}

    \subsection{Stability of Structure Coefficients}

    % In this section, we will not specialize the parameter $t$, and so we will use the plethystic notation $F[X]$ to mean $F(X;t)$. For our purposes, all we need to know about plethysm is that we can identify sets of variables with the sum of their elements, $X=x_1+x_2+\cdots$. Hence, the notation $X+Y$ is defined as the disjoint union of the sets $X=\{x_1,x_2,\ldots\}$ and $Y=\{y_1,y_2,\ldots\}$. In particular, the following identity tells us how to deal with the sum of alphabets.

    % \begin{proposition}[Sum Rule]
    %     We have
    %     \begin{align*}%p228
    %         Q_\lambda[X+Y]&=\sum_\mu Q_{\lambda/\mu}[X]Q_\mu[Y],\\
    %         B_\lambda[X+Y]&=\sum_\mu B_{\lambda/\mu}[X]B_\mu[Y].
    %     \end{align*}
    % \end{proposition}
    % For a function $F$ that is homogeneous of degree $n$, we have $F[-X]=(-1)^n(\omega F)[X]$. Hence, we have
    % \begin{align*}
    %     Q_\lambda[-X]&=(-1)^{|\lambda|}B_\lambda[X],\\
    %     B_\lambda[-X]&=(-1)^{|\lambda|}Q_\lambda[X].
    % \end{align*}
    % Now, we interpret the subtraction as $X-Y=X+(-Y)$. Thus, we can compute the actions of the operators $\alpha_z^\perp$ and $\beta_{-z}^\perp$ as follows.
    % \begin{proposition}
    %     Let $F(A;t)\in\Lambda(t)$, then
    %     \begin{align*}
    %         \alpha_z^\perp F[X]&=F[X+z],\\
    %         \beta_{-z}^\perp F[X]&=F[X-z].
    %     \end{align*}
    % \end{proposition}

    Now, we may use the vertex operator identity \eqref{eqn: vertex op action1} to prove a skew stability theorem. The following identities are necessary for the vertex operator method.

    First, the operators $\alpha_z^\perp$ and $\beta_{-z}^\perp$ act on a function $F\in\Lambda(t)$ via plethysm as follows.
    \begin{lemma}
        Let $F(A;t)\in\Lambda(t)$, then
        \begin{align*}
            \alpha_z^\perp F[X]&=F[X+z],\\
            \beta_{-z}^\perp F[X]&=F[X-z].
        \end{align*}
    \end{lemma}

    \begin{proof}
        It is sufficient to compute these for basis elements. From Proposition \ref{prop: action of adjoint} we have $\alpha_{z}^\perp\alpha_{z_1,z_2,\ldots}[X]=\alpha_{z_1,z_2,\ldots}[X+z]$, and so $\alpha_z^\perp Q_\lambda[X]=Q_\lambda[X+z]$. Similarly, we find that $\beta_{-z}^\perp Q_\lambda[X]=Q_\lambda[X-z]$.
    \end{proof}

    Next, we have an identity that allows one to separate a skew Hall-Littlewood function into a product of Hall-Littlewood functions.
    
    \begin{lemma}
        For all partitions $\lambda$ and integers $k,n\in\Z$ such that $k\geq0$ and $n>\lambda_1+k$, we have
        \[Q_{(n,\lambda)/(n-k)}=q_kQ_\lambda.\]
    \end{lemma}

    \begin{proof}
        Suppose a skew diagram is in two disconnected parts, say $\lambda/\mu=\gamma\oplus\delta$, where the Young diagrams $\gamma$ and $\delta$ do not share any edges or vertices. Then it is clear from the Young tableau formula \cite[p.~229]{macdonald:1995} that $Q_{\lambda/\mu}=Q_\gamma\cdot Q_\delta$.
    \end{proof}

    Lastly, we need the following identity to isolate the stability of our sequence.

    \begin{lemma}[\cite{graf-jing:2024b}]\label{lem: separate}
        Let $H(z)\in\Q[t][z,z^{-1}]$ be a Laurent polynomial. Then
        \begin{equation*}
            \frac{H(z)}{(1-z)}=L(z)+\frac{c(t)}{1-z},
        \end{equation*}
        where $c(t)\in\Q[t]$, and $L(z)$ is a Laurent polynomial. If $H(z)\neq0$, then $L(z)$ has degree at most $\max(\deg(H)-1,0)$.
    \end{lemma}

    Finally, we can prove the following stability theorem for the product of Hall-Littlewood functions.

    \begin{theorem}\label{thm: skew stability}
        Let $\lambda,\mu,\nu$ be partitions, then
        \begin{align*}
            \sum_{m,n\in\Z}\left(Q_{(m,\lambda)},Q_\mu Q_{(n,\nu)}\right)z^m&=L(z)+\frac{c(t)}{1-z},
        \end{align*}
        where $L(z)$ is a Laurent polynomial (with coefficients in $\Q[t]$) of degree at most $\mu_1+\nu_1+|\lambda|-|\mu|-|\nu|$, and $c(t)\in\Q[t]$. Similarly, we have
        \begin{align*}
            \sum_{m,n\in\Z}\left(B_{(m,\lambda)},B_\mu B_{(n,\nu)}\right)z^m&=H(z)+\frac{k(t)}{1-z}.
        \end{align*}
    \end{theorem}

    \begin{proof}
        It suffices to show only the first identity. Let $f(z)=\sum_{m,n\in\Z}\left(Q_{(m,\lambda)},Q_\mu Q_{(n,\nu)}\right)z^m$, then by \eqref{eqn: vertex op action1}, we have
        \begin{align*}
            f(z)&=\sum_{n\in\Z}\left(\sum_{m\in\Z}Q_{(m,\lambda)}z^m,Q_\mu Q_{(n,\nu)}\right)\\
            &=\sum_{n\in\Z}\left(\alpha_z\beta_{-1/z}^\perp Q_\lambda,Q_\mu Q_{(n,\nu)}\right)\\
            &=\sum_{n\in\Z}\left(\beta_{-1/z}^\perp Q_\lambda,\alpha_z^\perp Q_\mu Q_{(n,\nu)}\right).
        \end{align*}
        Now, we have $\alpha_z^\perp(Q_\mu Q_{(n,\nu)})=Q_\mu[X+z]Q_{(n,\nu)}[X+z]$. Note that $\beta_{-1/z}^\perp Q_\lambda=Q_\lambda[X-1/z]$ is a linear combination of terms of weight at most $|\lambda|$. Thus, in the expansion
        \begin{align*}
            \alpha_z^\perp Q_{(n,\nu)}&=Q_{(n,\nu)}[X+z]=\sum_{i\geq0}Q_{(n,\nu)/(i)}[X]q_i[z],
        \end{align*}
        we only need to consider terms where $|\lambda|\geq |\mu|+n+|\nu|-i$, i.e., where $i\geq n-(|\lambda|-|\mu|-|\nu|)$. Let $r=|\lambda|-|\mu|-|\nu|$, so that
        \begin{align*}
            f(z)&=\sum_{n\in\Z}\left(\beta_{-1/z}^\perp Q_\lambda,\alpha_z^\perp Q_\mu \sum_{i=n-r}^nQ_{(n,\nu)/(i)}q_i[z]\right).
        \end{align*}
        We write $f(z)=L(z)+T(z)$, where
        \begin{align*}
            L(z)&=\sum_{n\leq r}\left(\beta_{-1/z}^\perp Q_\lambda,\alpha_z^\perp Q_\mu Q_{(n,\nu)}\right),
        \end{align*}
        and
        \begin{align*}
            T(z)&=\sum_{n>r}\left(\beta_{-1/z}^\perp Q_\lambda,\alpha_z^\perp Q_\mu \sum_{i=n-r}^nQ_{(n,\nu)/(i)}q_i[z]\right).
        \end{align*}
        From \eqref{eqn: raising operator formulas}, it is clear that $Q_{(n,\nu)}=0$ if $n<-|\nu|$. Hence, $L(z)$ is a Laurent polynomial of degree at most $\nu_1+\mu_1+r$. Next, we may reindex the inner sum in $T(z)$ with $i\mapsto n-j$, and so we get
        \begin{align*}
            T(z)&=\sum_{n>r}\left(\beta_{-1/z}^\perp Q_\lambda,\alpha_z^\perp Q_\mu \sum_{j=0}^rQ_{(n,\nu)/(n-j)}q_{n-j}[z]\right).
        \end{align*}
        Note that $n-j\geq n-r>0$, and so $Q_{(n,\nu)/(n-j)}=q_jQ_\nu$ and $q_{n-j}[z]=(1-t)z^{n-j}$. Thus, we get
        \begin{align*}
            T(z)&=\sum_{n>r}\left(\beta_{-1/z}^\perp Q_\lambda,\alpha_z^\perp Q_\mu \sum_{j=0}^rq_jQ_\nu(1-t)z^{n-j}\right)\\
            &=\sum_{n>r}z^n\left(Q_\lambda[X-1/z],Q_\mu[X+z] \sum_{j=0}^rq_jQ_\nu(1-t)z^{-j}\right)\\
            &=\frac{z^{r+1}}{1-z}\cdot H(z),
        \end{align*}
        where $H(z)$ is a Laurent polynomial of degree at most $\mu_1$. Then, by Lemma \ref{lem: separate} we have
        \begin{align*}
            T(z)&=\frac{c(t)}{1-z}+K(z),
        \end{align*}
        where $K(z)$ is a Laurent polynomial of degree at most $\mu_1+r$.
    \end{proof}

    In other words, Theorem \ref{thm: skew stability} states that the sequences
    \begin{align*}
        (Q_{(m,\lambda)},Q_\mu Q_{(n,\nu)})&\qquad m\in\Z,n=m+|\lambda|-|\mu|-|\nu|,\\
        (B_{(m,\lambda)},B_\mu B_{(n,\nu)})&\qquad m\in\Z,n=m+|\lambda|-|\mu|-|\nu|,
    \end{align*}
    stabilize for large enough $m$. These sequences may equivalently be written as
    \begin{align*}
        (Q_{(m,\lambda)/\mu},Q_{(n,\nu)})&\qquad m\in\Z,n=m+|\lambda|-|\mu|-|\nu|,\\
        (B_{(m,\lambda)/\mu},B_{(n,\nu)})&\qquad m\in\Z,n=m+|\lambda|-|\mu|-|\nu|.
    \end{align*}

    \subsection{Stability of Hall Polynomials}

    For partitions $\lambda,\mu,\nu$, the Hall polynomial $g_{\mu\nu}^\lambda(t)$ and the coefficient $f_{\mu\nu}^\lambda(t)=(Q_{\lambda/\mu},Q_\nu)$ are related by the identity
    \begin{equation*}
        g_{\mu\nu}^\lambda(t)=t^{\varepsilon(\lambda)-\varepsilon(\mu)-\varepsilon(\nu)}f_{\mu\nu}^\lambda(t^{-1}),
    \end{equation*}
    where $\varepsilon(\lambda):=\sum_{i\geq1}\binom{\lambda_i'}{2}$ \cite[p.~217]{macdonald:1995}. Consider a partition of the form $\mu=(m,\lambda)$, where $m\geq\lambda_1$. Since $\mu_i'=\lambda_i'+1$ for all $1\leq i\leq m$, we have
    \begin{align*}
        \varepsilon(\mu)&=\sum_{i=1}^{\lambda_1}\binom{\lambda_i'+1}{2}+\sum_{i=\lambda_1+1}^m\binom{1}{2}=\sum_{i=1}^{\lambda_1}\binom{\lambda_i'+1}{2}.
    \end{align*}
    Hence $\varepsilon((m,\lambda))$ is constant for all $m\geq\lambda_1$, and in fact
    \begin{align*}
        \varepsilon((m,\lambda))&=\sum_{i=1}^{\lambda_1}\left[\binom{\lambda_i'}{1}+\binom{\lambda_i'}{2}\right]=|\lambda|+\varepsilon(\lambda)
    \end{align*}
    since $\binom{n}{k}=\binom{n-1}{k-1}+\binom{n-1}{k}$. Consequently, the stability of Hall polynomials follows from Theorem \ref{thm: skew stability}.
    \begin{theorem}
        For partitions $\lambda,\mu,\nu$, the following sequence of Hall polynomials stabilizes, 
        \[g_{\mu(n,\nu)}^{(m,\lambda)}(t),\qquad m\geq\lambda_1,~n=m+|\lambda|-|\mu|-|\nu|.\]
    \end{theorem}

    \section*{Acknowledgment}

    Thank you to Naihuan Jing for introducing me to this wonderful topic, and for offering me invaluable guidance along the way.

    \bibliographystyle{alpha}
    \bibliography{references}

    \bigskip
    \bigskip

    % \bigskip
    % \bigskip
    
    % \small
    
    % \noindent
    % J.G.:\newline
    % Department of Mathematics\\
    % North Carolina State University, Raleigh, NC 27695, USA\\
    % jrgraf@ncsu.edu
    
    % \vspace{5 mm}
    
    % \noindent
    % N.J.:\newline
    % Department of Mathematics\\
    % North Carolina State University, Raleigh, NC 27695, USA\\
    % jing@ncsu.edu

\end{document}